\documentclass[12pt,reqno]{amsart}

\newtheorem{lemma}{Lemma}[section]
\newtheorem{definition}{Definition}[section]
\newtheorem{proposition}{Proposition}[section]
\newtheorem{theorem}{Theorem}[section]
\newtheorem{remark}{Remark}[section]

\newtheorem{corollary}{Corollary}[section]
\newtheorem{example}{Example}[section]

\usepackage{amssymb,amsmath,amsthm}
\usepackage{epsfig,amssymb,amsmath,amsthm,graphicx,psfrag}
\usepackage[all]{xy}
\usepackage{enumerate}
\usepackage{tikz}
\usepackage{tikz-cd}
\usepackage{color}
\usepackage{array}
\usepackage{graphicx}
\usepackage{multirow}
\newcommand{\A}{\alpha}
\usepackage{hyperref}
\hypersetup{pdfborder=0 0 0}
\usetikzlibrary{calc}

\usetikzlibrary{arrows}

\usetikzlibrary{automata}

\usetikzlibrary{chains}

\makeatletter
\g@addto@macro{\endabstract}{\@setabstract}
\newcommand{\authorfootnotes}
{\renewcommand\thefootnote{\@fnsymbol\c@footnote}}%
\makeatother

\newcommand{\mypictureb}[1][]{
\begin{tikzpicture}[#1]

\def\dist{2.5}
\node at ($(-5.5,1.5)+(15:2.5)$) {$Q:$};

    \draw[fill=black] ($(0,0)+(90:2)$) circle (.08);

    \draw[fill=black] ($(0,0)+(0:2)$) circle (.08);

    \draw[fill=black] ($(0,0)+(45:2)$) circle (.08);
    
    \draw[fill=black] ($(0,0)+(270:2)$) circle (.08);

    \draw[fill=black] ($(0,0)+(225:2)$) circle (.08);
    
    \draw[fill=black] ($(0,0)+(135:2)$) circle (.08);

\draw[->,shorten <=7pt, shorten >=7pt] ($(-.3,4)+(280:2)$) arc (85:140:1.7);
\node at ($(-.4,.2)+(105:2)$) {$\alpha_0$};

\draw[->,shorten <=7pt, shorten >=7pt] ($(0,3)+(228:2.1)$) arc (135:180:2);
\node at ($(-.8,2.5)+(228.5:2.1)$) {$\alpha_1$};

\draw[->,shorten <=7pt, shorten >=7pt] ($(.5,.4)+(300:2)$) arc (325:355:2.8);
\node at ($(.1,1.7)+(338.5:2.5)$) {$\alpha_{k-1}$};
  
\draw[->,shorten <=7pt, shorten >=7pt] ($(-.35,0)+(280:2)$) arc (270:310:2.1);
\node at ($(.1,4.3)+(315-22.5:2.5)$) {$\alpha_k$};

\draw[->,shorten <=7pt] ($(1,-.6)+(78:2)$) arc (43:80:2.2);
\draw[->,shorten >=7pt] ($(.3,1.3)+(328:2)$) arc (5:50:1.7);
\draw[->,shorten <=7pt] ($(-1.9,2)+(270:2)$) arc (170:210:1.7);
\draw[->,shorten >=7pt] ($(.5,-.5)+(215:2)$) arc (240:270:2.3);
\node at ($(1.7,.2)+(270-20:2.5)$) {$\alpha_i$};
\node at ($(.15,-1.9)+(25:2.5)$) {$\alpha_{k-2}$};
\node at ($(-4.6,-1.5)+(20:2.5)$) {$\alpha_{i-2}\ $};
\node at ($(-1.9,.3)+(315-25:2.5)$) {$\ \alpha_{i-1}$};

\foreach \ang in {310,315,320,176,179.5,183}{
  \draw[fill=black] ($(0,0)+(\ang:2)$) circle (.02);
}
\end{tikzpicture}
}

\makeatletter
\tikzset{my loop/.style =  {to path={
  \pgfextra{}
  [looseness=12,min distance=10mm]
  \tikz@to@curve@path},font=\sffamily\small
  }}  
\makeatletter


\begin{document}



\title{UMP Monomial algebras: Combinatorial and Homological Consequences}

\author[Caranguay-Mainguez]{Jhony Caranguay-Mainguez}
\address{Universidad de Antioquia, Instituto de Matem\'aticas}
\email{jhony.caranguay@udea.edu.co}

\author[Franco]{Andr\'es Franco}
\address{Universidad de Antioquia, Instituto de Matem\'aticas}
\email{andres.francol@udea.edu.co}

\author[Reynoso-Mercado]{David Reynoso-Mercado}
\address{Universidad de Antioquia, Instituto de Matem\'aticas}
\email{david.reynoso@udea.edu.co}

\author[Rizzo]{Pedro Rizzo}
\address{Universidad de Antioquia, Instituto de Matem\'aticas}
\email{pedro.hernandez@udea.edu.co}






\maketitle



\begin{abstract}
In this paper, we apply the techniques developed in \cite{CFR} to present several consequences of studying UMP algebras and the ramifications graph of a monomial bound quiver algebra. Specifically, we prove that every weakly connected component of the ramifications graph of a UMP monomial algebra is unilaterally connected. Furthermore, using the main result characterizing UMP algebras in the monomial context, we prove that the class of UMP algebras is equivalent to the class of special multiserial algebras when the algebra is a quadratic monomial algebra. Based on this equivalence and the classification of Chen-Shen-Zhou on Gorenstein projective modules in \cite{chen2018gorenstein},  we extend their results to the class of monomial special multiserial  UMP algebras, where we use the analysis of homological properties on quadratic monomial algebras given by these authors.

\end{abstract}


\vskip5pt
\noindent
 
\vskip5pt
\noindent


\section{Introduction}
Gentle algebras (\cite{AS}), Almost Gentle algebras (\cite{Gr-Sc}), SAG and SUMP algebras (\cite{FGR1}) are examples of classes of bound quiver algebras satisfying the Unique Maximal Path property, briefly UMP algebras. That is, a UMP algebra is a bound quiver algebra such that two different maximal paths have no common arrows (are disjoint). The class of UMP algebras are introduced in \cite{CFR} in response to the problem proposed in \cite{FGR1} about the characterization of SUMP algebras. The authors in \cite{FGR1} def\mbox ine SUMP algebras as bound quiver algebras verifying a very special condition called the Unique Maximal Path property.
A complete classification is presented in \cite{CFR} for special multiserial and \emph{locally monomial} algebras (see Section 2 below or \cite{CFR} for the definition). In the context of bound quiver algebras, the Unique Maximal Path property appears to be crucial for the study of derived equivalences (see e.g. \cite{AG}, \cite{WX}). This property, along with those introduced above for certain derived categories, plays a significant role in classification problems of certain objects in important categories, as demonstrated in \cite{FM}, \cite{GRV} and \cite{Gr-Sc}.

In this paper, we present a collection of examples and non-examples of UMP algebras and explore the properties and implications of the combinatorial tools introduced in \cite{CFR} for a monomial bound quiver algebra $A$.
Indeed, we establish interesting connections between the mentioned tools and key combinatorial aspects of monomial bound quiver algebras. In this specific case, the set of \emph{$\omega$-relations} (see \cite[Def. 4.4]{CFR}) coincides with the minimal set of relations of the algebra (see \cite[Remark 4.2\,$iii)$]{CFR}). Consequently, the main classification result for UMP algebras is presented in a simplified and manageable version (see Section \ref{sec:UMPclass}), which directly impacts the study of their characterizations and homological properties. More precisely, by applying the techniques from \cite{CFR} and its main classification theorem for monomial algebras (see Theorem \ref{thm:UMPmain}), we characterize self-injective Nakayama algebras that are also UMP algebras. Furthermore, we prove that special multiserial algebras coincide with UMP algebras in the case of quadratic monomial algebras. This result motivated us to study the category of Gorenstein projective modules over monomial special multiserial UMP algebras. Specifically, the tools and results developed by Chen-Shen-Zhou in \cite{chen2018gorenstein} regarding Gorenstein projective modules for monomial algebras have a significant impact when applied to quadratic monomial algebras. For instance, they derive important classification results for various homological properties, such as a generalization of the Geiss-Reiten Theorem for gentle algebras (see \cite{GR}), using combinatorial expressions linked to the relations. Given that, in the quadratic monomial case, special multiserial algebras coincide with UMP algebras, and that Theorem \ref{thm:UMPmain} provides a complete characterization of this class in terms of relations, it is natural to extend the homological results from \cite{chen2018gorenstein} to the monomial special multiserial UMP case, as presented in Section \ref{sec:GprojUMP}.


This article is organized as follows. In Section 2, we begin by fixing some notations and we present some background material about bound quiver algebras and UMP algebras. Then, we state some results on UMP monomial algebras. Later, we introduce some examples showing the tools and techniques introduced in \cite{CFR}, finalizing with a comparative table exhibiting some properties of each example presented. In the last subsection of Section 2 we analyze some properties on the bound quiver algebras derived from its ramifications graph and vice versa. We will introduce here a new class of bound quiver algebras satisfying the property of the \emph{fully connected components} (see Definition \ref{def:fullyconn}), which contains the class of locally monomial and special multiserial algebras as well as the class of monomial and UMP algebras. In Section 3, we present the classification of some classes of UMP algebras which are relevant in representation theory of algebras such as Nakayama algebras and quadratic monomial algebras. Finally, in Section 4, we introduce the necessary preliminaries and tools to study the category of Gorenstein projective modules for monomial special multiserial UMP algebras, extending some classification results presented in \cite{chen2018gorenstein} for quadratic monomial algebras. Specifically, we provide a classification of perfect paths on UMP algebras (Theorem \ref{perfectarrows}) and, by generalizing the \emph{quiver of relations} from \cite{chen2018gorenstein} (Definition \ref{def:qrel}), we extend the main homological results to our context of UMP algebras (Theorem \ref{lema5.3}, Propositions \ref{prop:homoUMP1} and \ref{prop:homoUMP2}). We conclude this section with several examples demonstrating the application of these new tools and compare them with their counterparts in \cite{chen2018gorenstein}.


\section{UMP algebras: the monomial case}\label{Sec:UMP algebras}
For convenience of the reader we recall here some introductory material about bound quiver algebras and the preliminary material about UMP algebras introduced in \cite{CFR}. 

Let $A=\Bbbk Q/I$ be a finite-dimensional algebra over an algebraically closed field $\Bbbk$, where $I$ is an admissible ideal of $\Bbbk Q$ and $Q$ is a finite and connected quiver. Let $Q_0$ (resp. by $Q_1$) be the set of vertices (resp. the set of arrows) of $Q$. Also, $s(\alpha)$ (resp. $t(\alpha)$) denotes the vertex of $Q_0$ where the arrow $\alpha$ starts (resp. ends). We denote by $R$ a minimal set of relations such that $I=\langle R \rangle$. 

Each trivial path at vertex $i\in Q_0$ is denoted by $e_i$ and $\mathcal{P}(Q)$ denotes the set of all paths in $Q$. If $u,v\in \mathcal{P}(Q)$, we say that \emph{$u$ divides $v$} (or \emph{$u$ is a divisor of $v$}, or \emph{$u$ is a factor of $v$}, or \emph{$v$ factors through $u$}) if and only if $u$ is a subpath of $v$. We denote this relation by $u\mid v$. We also say that two paths are \textit{disjoint} if they have no common non-trivial divisors.

Let $\mathfrak{m}\in\mathcal{P}(Q)$. We say that $\mathfrak{m}+I$ is a \textit{maximal path} of $A=\Bbbk Q/I$ if $\mathfrak{m}\notin I$ and for every arrow $\alpha\in Q_1$ we have $\alpha \mathfrak{m}\in I$ and $\mathfrak{m}\alpha\in I$. We denote by $\mathcal{M}$ the set of maximal paths of $A$. We also say that two maximal paths $\mathfrak{m}+I$ and $\mathfrak{m}'+I$ are disjoint if for every pair of representatives $w$ of $\mathfrak{m}+I$ and $w'$ of $\mathfrak{m}'+I$, with $w,w'\in\mathcal{P}(Q)$, we have that $w$ and $w'$ are disjoint paths.

Given $w\in\mathcal{P}(Q)$, we write $w=w^0w^1\cdots w^{l_w}$ for the factorization of $w$ in terms of arrows $w^0,w^1,\dots, w^{l_w}\in Q_1$. In this case, we set $s(w):=s(w^0)$ and $t(w):=t(w^{l_w})$. Notice that the length of $w$ is $l_w+1$. 
A path is said to be \textit{repetition-free} if it has no repeated arrows as factors. 

An important tool introduced in \cite{CFR} is the \emph{ramifications graph}. Before defining it, however, we need to introduce the following prerequisite concept (see \cite[Definition 3.2]{CFR}). For each arrow $a\in Q_1$ we def\mbox ine a path $\omega_a$ in $\mathcal{P}(Q)$ as follows.
\begin{enumerate}[$i)$]
\item If $Q$ is a cyclic quiver, we f\mbox ix a repetition-free path of the form $\alpha_0\cdots \alpha_k$ in $\mathcal{P}(Q)$, where $\alpha_0,\ldots,\alpha_k$ are all the arrows of $Q$ (see e.g. Figure \ref{cycqu}). For each arrow $a\in Q_1$ we def\mbox ine $\omega_a:=\alpha_0\cdots \alpha_k$.
\item \item If $Q$ is not a cyclic quiver, we define $\omega_a$ in the following fashion: We define $\Phi_a$ as the set of all paths $\omega$ satisfying the following conditions.
\begin{enumerate}[$a)$]
\item $a$ is a subpath of $\omega$.
\item If $l_{\omega}>0$, then it holds that
\begin{equation*}\label{ineqs}
    |\{\alpha\in Q_1\mid s(\alpha)=i\}|= 1 \ \text{  and } \ |\{\alpha\in Q_1 \mid t(\alpha)=i\}|= 1
\end{equation*}
for every vertex $i\in\{t(\omega^j):0\leq j<l_{\omega}\}$.
\end{enumerate}

Then, we define $\omega_a$ as the path of maximal length in $\Phi_a$.
\end{enumerate}

\begin{figure}[h!]
\mypictureb[baseline=-22mm]
\caption{A cyclic quiver $Q$.}\label{cycqu}
\end{figure}

Roughly speaking, $\omega_a$ is the longest repetition-free path in $Q$ that contains the arrow $a$ and has no ramification vertices (vertices with at least two incoming or outgoing arrows). Now, following the notations of Definition 3.3 in \cite{CFR}, $G_{Q,I}:=(V,E)$ denotes the (oriented) \emph{ramifications graph associated to $(Q,I)$}, where the set of vertices is $V=\{\omega_a\mid a\in Q_1\}$ and for any pair $a,b\in Q_1$, there exists a directed edge $\delta\in E$ from $\omega_a$ to $\omega_b$ if and only if $\omega_a\neq\omega_b$, $t(\omega_a)=s(\omega_b)$ and $\omega_a^{l_a}\omega_b^{0}\notin I$. The set of \emph{weakly connected components of $G_{Q,I}$} is denoted by $\mathcal{D}_{Q,I}$, where a {\it weakly connected component} of $G_{Q,I}$ is a subgraph whose underlying undirected graph is a connected component of $G_{Q,I}$. Here, the underlying undirected graph is the graph obtained by ignoring the orientations of edges in $G_{Q,I}$.

For each $N\in \mathcal{D}_{Q,I}$, we denote by $Q_N$ the subquiver of $Q=(Q_0,Q_1,s,t)$ defined by the paths $\omega_a$ which are vertices in the component $N$. More precisely, $Q_N=((Q_N)_0,(Q_N)_1,s_N,t_N)$ where the set of arrows is defined by
$$
(Q_N)_1:=\{\alpha\in Q_1\mid \omega_{\alpha}\,\text{is a vertex in}\,N\},
$$
the set of vertices is defined by
$$
(Q_N)_0:=\{i\in Q_0\mid i\in\{s(\alpha),t(\alpha)\}\,\text{for some}\, \alpha \in (Q_N)_1\},
$$
and $s_N:=s|_{(Q_N)_1}$ and $t_N:=t|_{(Q_N)_1}$. From the subalgebra $\Bbbk Q_N$ of $\Bbbk Q$, we define $I_N$ as the induced ideal $I\cap\Bbbk Q_N$ of $\Bbbk Q_N$. With this, we define $A_N$ as the algebra $\Bbbk Q_N/I_N$ and $\mathcal{M}_N$ as the set of maximal paths in $A_N$. An important result, proved in \cite[Proposition 7]{CFR}, is that $I_N$ is an admissible ideal of $\Bbbk Q_N$, for all $N\in\mathcal{D}_{Q,I}$. In addition, it is not difficult to prove that if $A=\Bbbk Q/I$ is a monomial algebra, with $R$ a minimal set of zero relations generating $I$, then $A_N=\Bbbk Q_N/I_N$ is a monomial algebra for all $N\in\mathcal{D}_{Q,I}$, where $R_N:=R\cap\mathcal{P}(Q_N)$ is a minimal set of zero relations which generates $I_N$.

According to Remark 10 $i)$ in \cite{CFR}, if $u+I$ is a path in $A$, there exists a unique weakly connected component of $G_{Q,I}$, denoted by $N(u)$, such that $u\in \mathcal{P}(Q_{N(u)})$.

From \cite[Definition 15]{CFR}, a bound quiver algebra $A=\Bbbk Q/I$ is called a \emph{locally monomial} algebra if $A_N=\Bbbk Q_N/I_N$ is a monomial algebra for all $N\in\mathcal{D}_{Q,I}$. As pointed out above, the class of monomial algebras is contained in the class of locally monomial algebras.

Following \cite[Remark 10]{CFR}, for each $N\in\mathcal{D}_{Q,I}$ we can define a function $f_N:\mathcal{M}_N\rightarrow A$ by $\mathfrak{m}+I_N\mapsto \mathfrak{m}+I$. Also, we can define an equivalence relation ``$\sim_{\omega}$'' on the set of arrows $Q_1$ by $a\sim_{\omega} b$ if and only if $\omega_a + I=\omega_b + I$, for any $a,b\in Q_1$. Under these notations, we will prove a ``monomial'' version of Theorem 11 in \cite{CFR}.

\begin{proposition}\label{prop:maxUMP}
Let $A=\Bbbk Q/I$ be a bound quiver algebra and let $N\in\mathcal{D}_{Q,I}$. Then, the following statements hold.
\begin{enumerate}
\item For each $N\in\mathcal{D}_{Q,I}$, the function $f_N:\mathcal{M}_N\rightarrow A$ is injective.
\item If $A$ is a monomial algebra, then $\mathcal{M}=\bigsqcup\limits_{N\in\mathcal{D}_{Q,I}}f_N(\mathcal{M}_N)$.
\end{enumerate}
\end{proposition}
\begin{proof}
\begin{enumerate}
\item Suppose that there exist maximal paths $\mathfrak{m}+I_N$ and $\mathfrak{n}+I_N$ in $\mathcal{M}_N$, with $\mathfrak{m},\mathfrak{n}\in \mathcal{P}(Q_N)$, such that $f_N(\mathfrak{m}+I_N)=f_N(\mathfrak{n}+I_N)$. Then, $\mathfrak{m}+I=\mathfrak{n}+I$ and hence, $\mathfrak{m}-\mathfrak{n}\in I$. Also, since $\mathfrak{m}-\mathfrak{n}\in \Bbbk Q_N$ and $I_N=I\cap \Bbbk Q_N$, this implies that $\mathfrak{m}+I_N=\mathfrak{n}+I_N$.
\item In \cite[Theorem 11]{CFR}, it is proved that $\mathcal{M}=\bigcup_{N\in \mathcal{D}_{Q,I}}f_N(\mathcal{M}_N)$. Now, we will prove that this union is disjoint. Suppose that there are at least two weakly connected components $M$ and $N$ of $G_{Q,I}$ such that $f_N(\mathcal{M}_M) \cap f_N(\mathcal{M}_N) \neq \emptyset$. Then, there are paths $\mathfrak{m}$ in $\mathcal{P}(Q_M)$ and $\mathfrak{n}$ in $\mathcal{P}(Q_N)$ such that $\mathfrak{m}+I_M$ and $\mathfrak{n}+I_N$  are maximal paths in $A_M$ in $A_N$, respectively, and $\mathfrak{m}-\mathfrak{n}\in I$. Since $A$ is monomial algebra, then $I$ is generated by zero relations and hence, we have that $\mathfrak{m}=\mathfrak{n}$. So, $M=N(\mathfrak{m})=N(\mathfrak{n})=N$.

\end{enumerate}
\end{proof}

As an immediate consequence of Proposition \ref{prop:maxUMP}, we recover the UMP classification in the ``local case'', which is a special instance of the broader result presented in \cite{CFR}.
\begin{corollary}\label{cor:local1}
Let $A=\Bbbk Q/I$ be a monomial algebra. Then, 
\begin{enumerate}
\item  $A$ is an UMP algebra if and only if $A_N$ is a UMP algebra, for each $N\in \mathcal{D}_{Q,I}$.
\item $|\mathcal{M}|=\sum\limits_{N\in \mathcal{D}_{Q,I}}|\mathcal{M}_N|$.
\end{enumerate}

\end{corollary}

\begin{example}\label{ex1}
Consider the quiver $Q$ given by
$$
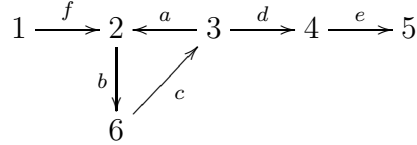
\xymatrix{
1 \ar[r]^f & 2 \ar[d]_b & \ar[l]_a 3 \ar[r]^d & 4 \ar[r]^e & 5\\
           &  6 \ar[ru]_c &                    &            &
}
$$
bounded by the set of relations $R=\{cabca, de, fb\}$. Then, we have $\mathcal{M}=\{f, abcabcd,e\}$ and $\omega_f=f,\ \omega_a=a,\ \omega_b=bc, \ \omega_d=de$. It follows that the ramifications graph $G_{Q,I}$ is composed by two weakly connected components,  $N: \omega_f$ and $N':\xymatrix{
\omega_a \ar[r] & \omega_b \ar[r] & \omega_d }$ and thus
\begin{figure}[!htb]
   \begin{minipage}{0.46\textwidth}
\centering
 $ Q_{N}:=\,\xymatrix{
1 \ar[r]^f & 2}$\\
\bigskip
$I_{N}=\langle0\rangle$,\ $\mathcal{M}_N=\{f\}$
   \end{minipage}\hfill
   \begin{minipage}{0.52\textwidth}
 \centering
$Q_{N'}:=\,\xymatrix{
2 \ar[d]_b  & \ar[l]_a 3 \ar[r]^d & 4 \ar[r]^e & 5\\
6 \ar[ru]_c &                    &            &}$\\
\bigskip
$I_{N'}=\langle cabca,de\rangle$,\ $\mathcal{M}_{N'}=\{abcabcd,e\}$
\end{minipage}
\end{figure}
\end{example}

\begin{example}\label{ex2}
Let $Q$ be the quiver given by
$$
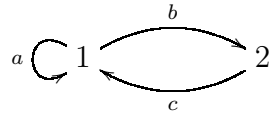
\xymatrix{
1 \ar@(ul,dl)[]_{a} \ar@/^1pc/[rr]^b
&& 2 \ar@/^1pc/[ll]^{c}}
$$
bounded by the set of relations $R=\{ab, ca, a^2-bc\}$. Thus, we have $\mathcal{M}=\{a^2(=bc),cb\}$ and $\omega_a=a$,  $\omega_b=\omega_c=bc$. Denoting by $N$ and $N'$ the weakly connected components defined by the paths $\omega_a$ and $\omega_b$, respectively, we have that: $N:\ \omega_a$, $N':\ \omega_b$ and
\begin{figure}[!htb]
   \begin{minipage}{0.48\textwidth}
\centering
 $ Q_{N}:=\,\xymatrix{
1 \ar@(ul,dl)[]_{a}}$\\
\bigskip
$I_{N}=\langle a^3\rangle$,\ $\mathcal{M}_{N}=\{a^2\}$
   \end{minipage}\hfill
   \begin{minipage}{0.48\textwidth}
 \centering
 $Q_{N'}:=\,\xymatrix{
1 \ar@/^1pc/[rr]^b
&& 2 \ar@/^1pc/[ll]^{c}}$\\
\bigskip
$I_{N'}=\langle bcb, cbc\rangle$,\ $\mathcal{M}_{N'}=\{bc,cb\}$
\end{minipage}
\end{figure}
\end{example}

\begin{example}\label{ex3}
Consider the quiver $Q$ given by
$$
Q:\ 
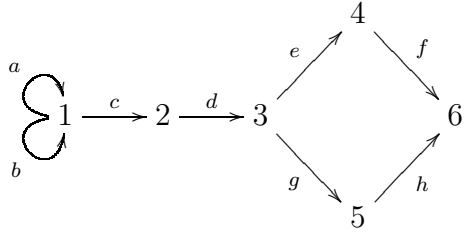
\xymatrix{ & & & 4 \ar[dr]^f & \\
1 \ar@(l,u)[]^{a} \ar@(l,d)[]_{b} \ar[r]^c & 2\ar[r]^{d} & 3 \ar[ur]^e \ar[dr]_g & & 6\\
 & & & 5 \ar[ur]_h &}
$$
\noindent
bounded by the set of relations $R=\{ab-ba, ef-gh, a^2, b^2, ac, bc, de\}$. Therefore, we obtain that $\mathcal{M}=\{ab(=ba), cdg,ef(=gh)\}$ and $\omega_a=a$,  $\omega_b=b$, $\omega_c=cd$, $\omega_e=ef$, $\omega_g=gh$. Denoting by $N$, $N'$ and $N''$ the weakly connected components we have that: $N:\,\xymatrix{
\omega_c \ar[r]
& \omega_g}$, $N':\,\omega_e$, $N'':\,\xymatrix{
\omega_a \ar@/^/[rr]
&& \omega_b \ar@/^/[ll]}$. 

\bigskip

The quiver, the ideal and the maximal paths in each case are given by
\medskip 

\begin{figure}[!htb]
\begin{minipage}{0.48\textwidth}
 \centering
 $Q_{N}:=\,\xymatrix{
1 \ar[r]^c
& 2 \ar[r]^{d} & 3 \ar[r]^{g}& 5 \ar[r]^{h} & 6}$\\
\bigskip
$I_{N}=\langle dgh\rangle$,\ $\mathcal{M}_{N}=\{cdg,gh\}$
\end{minipage}\hfill
   \begin{minipage}{0.48\textwidth}
 \centering
 $Q_{N'}:=\,\xymatrix{
3 \ar[r]^{e}& 4 \ar[r]^{f} & 6}$\\
\bigskip
$I_{N'}=\langle 0\rangle$,\ $\mathcal{M}_{N'}=\{ef\}$
\end{minipage}
\end{figure}

\begin{figure}[!htb]
\centering
 $ Q_{N''}:=\xymatrix{
1 \ar@(l,u)[]^{a} \ar@(l,d)[]_{b}
}$\\
\bigskip
$I_{N''}=\langle ab-ba,a^2,b^2\rangle$,\ $\mathcal{M}_{N''}=\{ab\}$
   \end{figure}

\end{example}

\begin{example}\label{ex4}
Consider the quiver $Q$ given by
$$
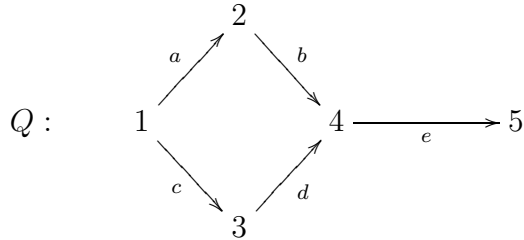
\xymatrix{& & 2\ar[dr]^{b} & & \\ 
Q: &1\ar[rd]_c \ar[ur]^{a}&  &4 \ar[rr]_{e} & & 5 \\ & & 3 \ar[ur]_{d} & & }
$$
bounded by the set of relations $R=\{be,de,ab-cd\}$. In this case, we obtain that $\omega_a=\omega_b=ab$, $\omega_c=\omega_d=cd$ and $\omega_e=e$. Denoting by $N$, $N'$ and $N''$ the weakly connected components we have:
\begin{figure}[!htb]
\begin{minipage}{0.48\textwidth}
 \centering
 $Q_{N}:=\,\xymatrix{
1 \ar[r]^c
& 3 \ar[r]^{d} & 4 }$\\
\bigskip
$I_{N}=\langle 0\rangle$
\end{minipage}\hfill
   \begin{minipage}{0.48\textwidth}
 \centering
  $Q_{N'}:=\,\xymatrix{
1 \ar[r]^a
& 2 \ar[r]^{b} & 4 }$\\
\bigskip
$I_{N'}=\langle 0\rangle$
\end{minipage}
\end{figure}

\begin{figure}[!htb]
\centering
  $Q_{N''}:=\,\xymatrix{
4 \ar[r]^e & 5 }$\\
\bigskip
$I_{N''}=\langle 0\rangle$
   \end{figure}

\end{example}


In the following table we present some properties for each algebra corresponding to the previous examples. Recall that, the bound quiver algebra $A=\mathbb{K}Q/I$ is a \emph{special multiserial algebra} if the following property holds: for every arrow $\alpha$ in $Q_1$ there is at most one arrow $\beta$ in $Q_1$ such that $\alpha\beta\notin I$ and at most one arrow $\gamma$ in $Q_1$ such that $\gamma\alpha\notin I$ (\cite[Definition 2.2]{GS}).

\begin{table}[ht]
    \centering
        \begin{tabular}{|c|c|c|c|c|}
            \hline 
            Example\textbackslash Properties & Special Multiserial & Monomial & Locally Monomial & UMP\\
            \hline 
            \ref{ex1} & $\times$ & \checkmark & \checkmark & \checkmark\\
            \hline 
            \ref{ex2} & \checkmark  & $\times$ & \checkmark & $\times$\\
            \hline 
            \ref{ex3} & \checkmark & $\times$& $\times$ & $\times$\\
            \hline 
           
   \ref{ex4} & \checkmark & $\times$& \checkmark  & \checkmark \\
            \hline 
        \end{tabular}
\end{table}

\subsection{Algebras whose ramifications graph is unilaterally connected}
A monomial algebra $A=\Bbbk Q/I$ which is also a UMP algebra, satisfies an important property related with the ramifications graph $G_{Q,I}$ associated to $A$, as stated in the following proposition. In this article, we will say that an oriented graph $G=(V,E)$ is \emph{unilaterally connected} if for any pair of vertices $i,j\in V$, it contains an oriented path from $i$ to $j$ or from $j$ to $i$.
\begin{theorem}\label{thm:fullyconn}
Let $A=\Bbbk Q/I$ be a monomial algebra. If $A$ is a UMP algebra, then every weakly connected component of $G_{Q,I}$ is unilaterally connected.
\end{theorem}
\begin{proof}
Let $N\in\mathcal{D}_{Q,I}$ be a weakly connected component and, by contradiction, suppose that $N$ is not unilaterally connected. Then, there exist $a,b\in Q_1$ such that $\omega_a$ and $\omega_b$ are two vertices in $G_{Q,I}$ which are not connected by any oriented path in $N$. Note that they are connected by a path in the undirected underlying graph of $N$. Thus, there exists a sequence $\omega_{a_1},\omega_{a_2},\ldots,\omega_{a_n}$ of vertices in $N$, with $a_i\in Q_1$, $a_1=a$ and $a_n=b$, such that there exists an edge linking $\omega_{a_i}$ and $\omega_{a_{i+1}}$, for each $i\in\{1,\ldots,n-1\}$. Let $k\in\mathbb{Z}^+$ be the least integer such that $\omega_a$ and $\omega_{a_k}$ are not unilaterally connected. Hence, $k>2$ and there is an oriented path $\delta$ in $N$ from $\omega_a$ to $\omega_{a_{k-1}}$ or viceversa. Suppose, without loss of generality, that $\delta$ is an oriented path from $\omega_a$ to $\omega_{a_{k-1}}$, whose factorization by edges in $E$ is $\delta=\delta^1\cdots\delta^d$. Let $b_i$ be an arrow such that $\omega_{b_i}=t(\delta^i)$ for $1\leq i\leq d-1$, and we set $b_0:=a$ and $b_d:=a_{k-1}$. Furthermore, since $N$ is not a unilaterally connected graph, there exists a unique edge $\epsilon$ in $E$ from $\omega_{a_{k}}$ to $\omega_{b_d}$. Thus, $\omega_{a_k}^{l_{a_k}}\omega_{b_d}^0\notin I$ and $\omega_{b_{d-1}}^{l_{b_{d-1}}}\omega_{b_d}^0\notin I$. Due to the fact that $A$ is a UMP monomial algebra, there is a unique maximal path $\mathfrak{m}+I\in\mathcal{M}$ such that $\mathfrak{m}$ contains $\omega_{a_k}^{l_{a_k}}\omega_{b_d}^0$ and $\omega_{b_{d-1}}^{l_{b_{d-1}}}\omega_{b_d}^0$ as factors. It follows that $\mathfrak{m}=x\omega_{b_{d-1}}^{l_{b_{d-1}}}\omega_{b_d}^0y\omega_{a_k}^{l_{a_k}}\omega_{b_d}^0z$ or $\mathfrak{m}=x\omega_{a_k}^{l_{a_k}}\omega_{b_d}^0 y\omega_{b_{d-1}}^{l_{b_{d-1}}}\omega_{b_d}^0 z$, for some $x,y,z\in\mathcal{P}(Q)$. 

Suppose that $\mathfrak{m}=x\omega_{b_{d-1}}^{l_{b_{d-1}}}\omega_{b_d}^0y\omega_{a_k}^{l_{a_k}}\omega_{b_d}^0z$. Since $\omega_{b_d}$ and $\omega_{a_k}$ satisfy the conditions in (\ref{ineqs}) for the arrows $a_{k-1}$ and $a_k$, respectively, we obtain that $\omega_{b_d}\in E_{\omega_{b_d}^0y}$ and $\omega_{a_k}\in T_{y\omega_{a_k}^{l_{a_k}}}$. Now, since $\omega_{b_d}^0y\omega_{a_k}^{l_{a_k}}\notin I$, we conclude that there is a path from $\omega_{b_d}$ to $\omega_{a_k}$, which is a contradiction with the choice of $k$. Hence, $\mathfrak{m}=x\omega_{a_k}^{l_{a_k}}\omega_{b_d}^0 y\omega_{b_{d-1}}^{l_{b_{d-1}}}\omega_{b_d}^0 z$. This implies that the set of non--zero paths in $A$ of the form $\omega_{a_k}^{l_{a_k}}u\omega_{b_j}^{l_{b_j}}$, with $j<d$, is non--empty. Let $j\in\mathbb{Z}$ be the least integer with the latter condition. Then, $j>0$. On the contrary, $\omega_{a_k}^{l_k}u\omega_{b_0}^{l_{b_0}}$ induce a path in $N$ from $\omega_{a_k}$ to $\omega_a=\omega_{b_0}$, which contradicts the choice of $\omega_{a_k}$. Since $j\neq d$, we have by conditions in (\ref{ineqs}) that $\omega_{b_d}\in E_{\omega_{b_d}^0u}$ and $\omega_{b_j}\in T_{u\omega_{b_j}^{l_{b_j}}}$. Hence, $\omega_{b_j}^0$ divides both $\omega_{a_k}^{l_{a_k}}u\omega_{b_j}^{l_{b_j}}$ and $\omega_{b_{j-1}}^{l_{b_{j-1}}}\omega_{b_j}^0$. Furthermore, observe that the choice of $j$ implies that $u^{l_u}\neq w_{b_{j-1}}^{l_{b_{j-1}}}$. Now, since $A$ is a UMP algebra, there exists a unique maximal path $\mathfrak{m}'$ such that either $\mathfrak{m}'=x'\omega_{a_k}^{l_{a_k}}u\omega_{b_j}^{l_{b_j}}y'\omega_{b_{j-1}}^{l_{b_{j-1}}}\omega_{b_j}^0z'$ or $\mathfrak{m}'=x'\omega_{b_{j-1}}^{l_{b_{j-1}}}\omega_{b_j}^0 y'\omega_{a_k}^{l_{a_k}}u\omega_{b_j}^{l_{b_j}}z'$, for some $x', y', z'\in\mathcal{P}(Q)$. However, the former case cannot occur since our assumption on $j$ implies that $\omega_{a_k}^{l_{a_k}}u\omega_{b_j}^{l_{b_j}}y'\omega_{b_{j-1}}^{l_{b_{j-1}}}\in I$. Therefore, $\mathfrak{m}'=x'\omega_{b_{j-1}}^{l_{b_{j-1}}}\omega_{b_j}^0 y'\omega_{a_k}^{l_{a_k}}u\omega_{b_j}^{l_{b_j}}z'$. 

Now, since $\omega_{b_j}\neq\omega_{a_k}$ we have that $\omega_{a_k}\in T_{y'\omega_{a_k}^{l_k}}$. Hence, $\mathfrak{m}'$ defines an oriented path in $N$ from $\omega_{b_j}$ to $\omega_{a_k}$, which contradicts the choice of the arrow $a_k$. This completes the proof.
\end{proof}

The property of the ramifications graph $G_{Q,I}$ in which every weakly connected component is also unilaterally connected, plays an important role for several classes of algebras as in the following discussion. We highlight this property in the next definition.

\begin{definition}\label{def:fullyconn}
Let $A=\Bbbk Q/I$ be a bound quiver algebra. We say that $A$ has \emph{fully connected components} if each weakly connected component of $G_{Q,I}$ is a unilaterally connected graph.
\end{definition}

Theorem \ref{thm:fullyconn} implies that every monomial UMP algebra has fully connected components. However, the converse is not true. That is, there exist monomial algebras having fully connected components which are not UMP algebras. Indeed, consider the monomial algebra $A$ defined by the quiver 
$$
Q:\ 
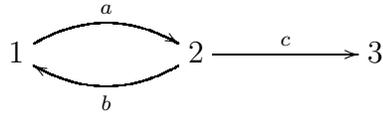
\xymatrix{
1  \ar@/^1pc/[rr]^a
&& 2 \ar@/^1pc/[ll]^{b} \ar[rr]^{c} && 3}
$$
bounded by the set of relations $R=\{aba,bab\}$. Then, $A$ is not a UMP algebra since the set of maximal paths is given by $\mathcal{M}=\{bac,ab\}$. Nevertheless, $\omega_a=\omega_b=ba$, $\omega_c=c$ and $G_{Q,I}$ has a unique weakly connected component given by
$$
N:\, \xymatrix{
\omega_a \ar[rr] && \omega_c}
$$
which is a unilaterally connected graph.

However, note that if $A$ is a special multiserial algebra, then $A$ has fully connected components with an additional feature. More precisely, in this case every weakly connected component has one of the following forms, for certain $n\in\mathbb{N}$:
\begin{figure}[!htb]
   \begin{minipage}{0.48\textwidth}
\centering
 $ N:=\,\xymatrix{
\omega_0 \ar[r] & \omega_1\ar[r]  & \cdots\ar[r] & \omega_n}$
   \end{minipage}\hfill
   \begin{minipage}{0.48\textwidth}
 \centering
 $N:=\,\xymatrix{
\omega_0 \ar[r] & \omega_1\ar[r]  & \cdots\ar[r]
& \omega_n \ar@/^2pc/[lll]}$
\end{minipage}
\end{figure}

\vskip 0.3cm

We distinguish this additional graph feature by saying that the corresponding component is of \emph{Nakayama type}. We use this name by its similarity with the classification of finite dimensional Nakayama algebras. It is worth mentioning that this feature is essential in the classification of the special multiserial UMP algebras. More details in \cite{CFR}.

A natural question is whether the features of the graph $G_{Q,I}$ determine specific properties of the corresponding algebra $A$. In this regard, we can announce.

\begin{proposition}
Let $A=\Bbbk Q/I$ be a bound quiver algebra. If each weakly connected component of $G_{Q,I}$ is of Nakayama type and none of its vertices corresponds to a cycle in $Q$, then $A$ is a special multiserial algebra.
\end{proposition}
\begin{proof}
Let $a,b,c$ be arrows in $Q_1$, with $b\neq c$ and $s(b)=s(c)=t(a)$. Suppose, by contradiction, that $ab\notin I$ and $ac\notin I$. In this case, $\omega_a^{l_a}=a$, $\omega_b^0=b$ and $\omega_c^0=c$. Since each weakly connected component of $G_{Q,I}$ is of Nakayama type, then there exists at most one edge in $G_{Q,I}$ starting at $\omega_a$. Hence, either $\omega_a=\omega_b$ or $\omega_a=\omega_c$. Suppose, without loss of generality, that $\omega_a=\omega_b$. However, this is impossible because $t(\omega_a)=t(a)=s(b)=s(\omega_b)=s(\omega_a)$ and therefore, $\omega$ is a cycle, which is a contradiction with the hypothesis. A similar reasoning applies in the case in which $b\neq c$ and $s(a)=t(b)=t(c)$, for some arrows $a,b,c$ in $Q_1$.
\end{proof}
\begin{remark}
The hypothesis that none of the vertices of $G_{Q,I}$ correspond to a cycle in $Q$ is essential for obtaining that $A$ is a special multiserial algebra. Indeed, the following quiver with  $I=\langle abc\rangle$, defines the algebra $A=\Bbbk Q/I$ which is not a special multiserial algebra, and the unique weakly connected component in $G_{Q,I}$ is of Nakayama type.
\begin{figure}[!htb]
   \begin{minipage}{0.48\textwidth}
\centering
 $ Q:\, \xymatrix{
2  \ar[dr]_{c}
& 1 \ar[l]_b \\
& 3 \ar[u]_a\ar[r]_d & 4 }$
   \end{minipage}\hfill
   \begin{minipage}{0.48\textwidth}
 \centering
 $G_{Q,I}:\, \xymatrix{
\omega_a \ar[rr] && \omega_d}$
\end{minipage}
\end{figure}

\end{remark}

\section{UMP monomial algebras: special cases of classification}\label{sec:UMPclass}
In this section we will discuss the classification of some classes of UMP monomial algebras which are relevant in representation theory such as Nakayama algebras and quadratic monomial special multiserial algebras.

For the sake of completeness, we present the monomial version of the main theorem in \cite{CFR} (see \cite[Corollary 4.1]{CFR}) which we will apply in the two cases of classification mentioned above.

\begin{theorem}\label{thm:UMPmain}
Let $A=\mathbb{K}Q/I$ be a monomial special multiserial algebra and let $R$ be a minimal set of relations such that $I=\langle R \rangle$. Then, $A$ is a UMP algebra if and only if for any zero relation $r\in R$ with length greater than two, there exists a path $u$, such that $t(u)=s(u)$ and 
\begin{equation*}\label{eq:main}
   R\cap\{\text{subpaths of}\,\, u^{(p)}:\,p\in\mathbb{Z}^+\}=\{r\}. 
\end{equation*}
where $u^{(p)}$ denotes the composition of the path $u$ with itself $p$-times.
\end{theorem}

\subsection{Nakayama algebras}
In this subsection we classify the selfinjective Nakayama algebras which are also UMP algebras. 

Let $A=\Bbbk Q/I$ be a bound quiver algebra. It is well known that $A$ is a selfinjective Nakayama algebra if and only if $A=N_n^m(\Bbbk)$, for some positive integers $m$ and $n$ (see \cite[Thm. 6.15, p. 384]{SY}). Here, $N_n^m(\Bbbk)=\Bbbk\Delta_n/I_{n,m}$ is the associated bound quiver algebra to the cyclic quiver:
$$
\Delta_n\,:=\,\xymatrix{
1 \ar[r]^{\alpha_1} & 2\ar[r]^{\alpha_2}  & 3 \ar[r]^{\alpha_3} &\cdots\ar[r]^{\alpha_{i-1}} & i \ar[r]^-{\alpha_i} & \cdots n-1 \ar[r]^-{\alpha_{n-1}} & n \ar@/^3pc/[llllll]^{\alpha_n}
}
$$
and $I_{n,m}$ is the admissible ideal of the path algebra $\Bbbk\Delta_n$ generated by all compositions of $m+1$ consecutive arrows in $\Delta_n$. As a direct consequence of Theorem \ref{thm:UMPmain} we obtain the following proposition.
\begin{proposition}\label{prop:selfinj}
The selfinjective Nakayama algebra $A=N_n^m(\Bbbk)$ is a UMP algebra if and only if $A=N_n^1(\Bbbk)$. 
\end{proposition}

Nevertheless, following the classification Theorem in \cite[Thm. 10.3, p. 102]{SY}, there exist non-selfinjective Nakayama algebras $A$ which are UMP algebras. For example, consider the quiver
$$
Q:\,\xymatrix{                                  
              &3\ar[dl]_c &  \\
1 \ar[rr]_a &           &2\ar[ul]_b
}
$$
with $I=\langle abc\rangle$. The bound quiver algebra $A=\Bbbk Q/I$ is a Nakayama algebra which is a UMP algebra by Theorem \ref{thm:UMPmain} and is not a selfinjective Nakayama algebra by Proposition \ref{prop:selfinj}.

\subsection{Quadratic monomial special multiserial algebras}
As an immediate consequence of Theorem \ref{thm:UMPmain}, any quadratic monomial special multiserial algebra is a UMP algebra, which was firstly proved in \cite[Lemma 12]{FGR1}. The following result establishes that the converse also holds.

\begin{theorem}\label{thm:umpquadratic}
Let $A=\Bbbk Q/I$ be a quadratic monomial algebra and let $R$ be a minimal set of relations such that $I=\langle R \rangle$. Then $A$ is a UMP algebra if and only if $A$ is a special multiserial algebra.
\end{theorem}
\begin{proof}
If $A$ is a special multiserial algebra, then by Theorem \ref{thm:UMPmain}, $A$ is a UMP algebra. Conversely, suppose that $A$ is a UMP algebra and, by contradiction, suppose also that $A$ is not a special multiserial algebra. Then, we have two cases: 
\begin{enumerate}
\item There exist $\alpha,\beta,\gamma\in Q_1$, with $\beta\neq\gamma$ and $s(\beta)=s(\gamma)=t(\alpha)$, such that $\alpha\beta\notin I$ and $\alpha\gamma\notin I$.
\item There exist $\alpha,\beta,\gamma\in Q_1$, with $\alpha\neq\beta$ and $t(\alpha)=t(\beta)=s(\gamma)$, such that $\alpha\gamma\notin I$ and $\beta\gamma\notin I$.
\end{enumerate}
 In the case $(1)$, since $A$ is a UMP algebra, there exists a maximal path of the form either $x\alpha\beta x'\alpha\gamma x''$ or $x\alpha\gamma x'\alpha\beta x''$, for certain $x,x',x''\in\mathcal{P}(Q)$. In the former situation, $\beta x'\alpha$ is a cycle, and hence there exists $r\in R$ such that $r\mid(\beta x'\alpha)^{(l)}$, for some $l\in\mathbb{Z}^+$, because $I$ is an admissible ideal. Due to the fact that $A$ is a quadratic monomial algebra, the length of $r$ is two. Now, since $\alpha\beta\notin I$, we have $r\mid\beta x'\alpha$, which contradicts the maximality of the path $x\alpha\beta x'\alpha\gamma x''$. An analogous reasoning applies for the latter situation. Similarly, for the case $(2)$ we also obtain a contradiction. Thus, $A$ is a special multiserial algebra.
\end{proof}

\section{Homological consequences}\label{sec:GprojUMP}
Inspired by Theorems \ref{thm:UMPmain} and \ref{thm:umpquadratic} and the study of Gorenstein projective modules over quadratic monomial algebras presented in \cite{chen2018gorenstein}, this section is dedicated to extending the tools and results of \cite{chen2018gorenstein} to the UMP context. Our main findings are presented in subsection \ref{subsec:mainfind}, with preceding subsections introducing essential tools and results from a UMP perspective, thus demonstrating the naturality of our generalizations.

\subsection{About minimal sets of relations}
We start by providing an essential background for understanding the main results and subsequent computations. For each vertex $i\in Q_0$, we define the sets:
\begin{eqnarray}\label{eq:pm}
    i^+=\{\alpha\in Q_1\mid\, s(\alpha)=i\}&\mbox{and}&i^-=\{\alpha\in Q_1\mid\,t(\alpha)=i\}.
\end{eqnarray}

Let $A=\Bbbk Q/I$ be a monomial algebra, and $R$ be a minimal set of relations such that $I=\langle R \rangle$. For a nonzero path $p$, i.e., $p$ does not contain a subpath in $R$, we denote by $Ap$ (resp. $pA$) the left (resp. the right) ideal generated by $p$. In this case, the basis of $Ap$ (resp. $pA$) are nonzero paths $q$ such that $q=rp$ (resp. $q=pr$), for some path $r$. 

For any subset $\mathbf{S}$ of $\mathcal{P}(Q)$ we say that a path $p$ in $\mathbf{S}$ is left-minimal (resp. right-minimal) in $\mathbf{S}$ provided that there is no path $q\in \mathbf{S}$ such that $p = p'q$ (resp. $p=qp'$) for some non-trivial path $p'$. Following \cite{chen2018gorenstein}, we define $\mathcal{R}(p)$ (resp. $\mathcal{L}(p)$) the set of right-minimal (resp. left-minimal) paths in the set
$$
\left\{\text{nonzero paths}\, q\mid s(q)=t(p)\,\text{and}\,pq=0\right\}\,(\text{resp.} \left\{\text{nonzero paths}\, q\mid t(q)=s(p)\,\text{and}\,qp=0\right\})
$$

\begin{remark}\label{rmk:RLsets}
If $A$ is a monomial and special multiserial algebra, we have the following classification of the sets $\mathcal{R}(p)$ and $\mathcal{L}(p)$, for any nonzero path $p$: If $q\in \mathcal{R}(p)$ (resp. $q\in \mathcal{L}(p)$) then $q$ is either an arrow or the unique path of length greater than 1.

In fact, if $q\in \mathcal{R}(p)$ is not an arrow, then $p^{l_p}q^0\neq 0$. Consequently, for any $r\in \mathcal{R}(p)$ with $r$ distinct from $q$, it follows that $p^{l_p}r^0= 0$, since $A$ is a special multiserial algebra. Thus, $r=r^0$. A similar argument applies to $\mathcal{L}(p)$.
\end{remark}

\begin{remark}\label{remark_descripcion_relaciones}
An immediate consequence from Theorem \ref{thm:UMPmain} for the minimal set of relations $R$ of a monomial, special multiserial UMP algebra $A$ is as follows: If $p_1,\,p_2$, $p_3$ are nonzero paths such that $p_1p_2,\,p_2p_3\in R$ are relations with length greater than two, then $p_1p_2=p_2p_3$ as paths in $Q$.
\end{remark}

\begin{lemma}\label{lemma3.14}
Let $A=\Bbbk Q/I$ be a monomial special multiserial UMP algebra, and let $R$ be a minimal set of relations such that $I=\langle R \rangle$. Consider $p\in R$ for which there exists a cycle $u=u^0\cdots u^{l_u}$, with $l_u\geq 1$, such that $R\cap\{\text{subpaths of}\,\, u^{(l)}:\,l\in\mathbb{Z}^+\}=\{p\}$.
\begin{enumerate}
\item Let $q$ and $r$ be nonzero paths with $t(r)=s(q)$ (resp. $s(r)=t(q)$) and such that $rq=0$ (resp. $qr=0$). If $q$ is disjoint from $u$ and, for some $l$, $r\mid u^{(l+1)}$, then $r^{l_r}q^0\in R$ (resp. $q^{l_q}r^0\in R$).
\item If $q$ is a nonzero path sharing a subpath with $u^{(m)}$ for some $m$, then $q$ is a subpath of $u^{(m)}$. 
\end{enumerate}
\end{lemma}
 \begin{proof}
First, note that $l_p\geq2$ by Remark \ref{remark_descripcion_relaciones}, and by Theorem \ref{thm:UMPmain}, $p$ is the unique relation appearing in the cycle $u$. Thus, for the first claim, if $rq=0$, we have $r^{l_r}u^j\neq 0$ for some $0\leq j\leq l_u$. Since $A$ is special multiserial, it follows that $r^{l_r}q^0\in R$. The other case follows similarly.

For the second claim, suppose that $q=q'vq''$, where $v$ a nonzero subpath of $u^{(m)}$. Assume that $q'$ and $q''$ are non-trivial paths such that $q'\nmid u$ and $q''\nmid u$. Since $l_p\geq2$ and $A$ is an UMP algebra, it follows that $v^{l_v}u^j\neq 0$, for some $1\leq j\leq l_u$. Consequently, $v^{l_v}(q'')^{0}=0$, as $A$ is special multiserial. Then, $q=0$, which contradicts the hypothesis. Thus, $q''$ must be a trivial path.  By a similar argument, $q'$ is also a trivial path. Hence, $q=v$ and $q$ is a subpath of $u^{(m)}$.
\end{proof}

Now, if $A$ is a monomial, special multiserial UMP algebra, we can provide a more detailed description of the sets $\mathcal{R}(p)$ and $\mathcal{L}(p)$, for any nonzero path $p$ (cf. Remark \ref{rmk:RLsets}).

\begin{lemma}\label{lemma_LR}
Let $A=\Bbbk Q/I$ be a monomial special multiserial UMP algebra, and let $R$ be a minimal set of relations such that $I=\langle R \rangle$. For a nonzero path $p$, we have $q\in \mathcal{R}(p)$ (resp. $q\in \mathcal{L}(p)$) if and only if $rq\in R$ (resp. $qr\in R$), for some non-trivial path $r$ with $p=p'r$ (resp. $p=rp'$).
\end{lemma}
\begin{proof}
Suppose that $q\in \mathcal{R}(p)$, that is, $s(q)=t(p)$, $pq=0$ and there no exists $q'\in \mathcal{R}(p)$ such that $q=q'q''$, for some non-trivial path $q''$. Since $pq=0$, we have that $pq=p'rq'$ where $r\in R$. In particular, by $p$ and $q$ be nonzero paths, $r\nmid p$ and $r\nmid q$, i.e., $q=r_2q'$ and $p=p'r_1$, with $r_1, r_2$ proper subpaths of $r$. In consequence, $pr_2=p'r=0$ but $q=r_2q'$, which implies that $q'$ is a trivial path. Thus, $q=r_2$, which proves this part of the lemma. Conversely, suppose that $q$ is a nonzero path such that $rq\in R$, for some path $r$ with $p=p'r$. With loss of generality, we can suppose that $q$ is of the form $q=q'q''$, with $pq'=0$. We analyze two cases: when the length of $rq$ is equal to two and when it is greater than two. If the length of $rq$ is equal to two, then $q$ must be an arrow. Since $p$ is a nonzero path, $q'$ is non-trivial. Therefore, $q=q'$ and $q\in \mathcal{R}(p)$. If the length of $rq$ is greater than two, by Theorem \ref{thm:UMPmain}, there exists a cycle $u$ such that $rq$ satisfies the conclusion. Now, if $pq'=0$, then we can write $pq'=r_1sr_2$ for some $s\in R$. Since $p$ and $q'$ share a subpath with some power of $u$, Lemma \ref{lemma3.14} implies that they are completely contained within this power. Consequently, $s=rq$ and thus $q\in \mathcal{R}(p)$.
\end{proof}

A motivation for studying the sets $\mathcal{R}(p)$ and $\mathcal{L}(p)$ for any nonzero path $p$ lies in the well-known general classification of modules over a monomial algebra $A$.

\begin{lemma}\label{lemma3.1} Let $p$ be a non-zero and non-trivial path in $A$. Then, there exists the following exact sequence of left $A$-modules:
\begin{eqnarray}
    0\rightarrow\bigoplus_{q\in \mathcal{L}(p)}Aq\overset{\iota}{\longrightarrow}Ae_{s(p)}\overset{\pi_p}{\longrightarrow}Ap\rightarrow0
\end{eqnarray}
where $\iota$ is the inclusion map and $\pi_p$ is the projective cover of $Ap$, where $\pi_p(e_{s(p)})=p$. Similarly, we have the following exact sequence of right $A$-modules:
\begin{eqnarray}
    0\rightarrow\bigoplus_{q\in \mathcal{R}(p)}qA\overset{\iota}{\longrightarrow}Ae_{t(p)}\overset{\pi'_p}{\longrightarrow}pA\rightarrow0
\end{eqnarray}
where $\pi'_p$ is the projective cover of $pA$, where $\pi'_p(e_{t(p)}) =p$.
\end{lemma}
\begin{proof}
Cf. the first paragraph of \cite[p. 162]{Zimm}. 
\end{proof}

\begin{lemma}\label{lemma3.2} Let $M$ be a left $A$-module that fits into an exact sequence of $A$-modules:
$$0\rightarrow M\rightarrow P\rightarrow Q$$
with $P$, $Q$ projective. Then $M$ is isomorphic to a direct sum $\bigoplus_p Ap^{(\Lambda(p))}$, where $p$ runs through all nonzero paths in $A$ and each $\Lambda(p)$ is some index set.    
\end{lemma}
\begin{proof}
\cite[Theorem I]{Zimm}.   
\end{proof}

\subsection{Perfect pairs on monomial UMP algebras}
In \cite{chen2018gorenstein}, perfect pairs play a crucial role in the study of Gorenstein projective modules over monomial algebras. In this section, we study properties and derive results for the special case of monomial special multiserial algebras, with a particular focus on the subclass of UMP algebras. For the sake of completeness, we recall the definition of a perfect pair, with the necessary modifications for our specific context.

\begin{definition}\label{def:perpair}
Let $A = \Bbbk Q/I$ be a monomial algebra. A pair $(p,q)$ of nonzero paths in $A$ is \emph{perfect} if the following conditions are satisfied:
    \begin{enumerate}
        \item[(P1)] both nonzero paths $p,\,q$ are non-trivial, satisfying $t(p) = s(q)$ and $pq = 0$ in $A$;
        \item[(P2)] if $q'$ is a nonzero path, with $s(q') = t(p)$ and $pq' = 0$ then, $q' = qq''$, for some path $q''$, that is, $\mathcal{R}(p) = \{q\}$;
        \item[(P3)] if $p'$ is a nonzero path with $t(p') = s(q)$ and $p'q = 0$ then, $p' = p''p$ form some path $p''$, that is,  $\mathcal{L}(q) = \{p\}$. 
    \end{enumerate}
\end{definition}

\begin{lemma}\label{lemaSM_perfectpair}
Let $A=\Bbbk Q/I$ be a monomial special multiserial algebra, with $R$ minimal set of relations. For a perfect pair $(p,\,q)$ with $l_{pq}>1$, we have that the sets (cf. (\ref{eq:pm})) $|t(\alpha)^-|=1=|s(\beta)^+|$, for any arrows $\alpha$ dividing $p$ and $\beta$ dividing $q$.
\end{lemma}
\begin{proof}
Let $A$ be a a monomial special multiserial algebra. Let $\alpha$ be an arrow that divides $p$. Suppose that $t(\alpha)=t(p)$. If $\alpha q^{0}\in R$, we have $\alpha q=\delta\gamma$, where $\delta=\alpha q^{0}\in R$. Since $(p,q)$ is a perfect pair, $\alpha=p$ implying $pq^{0}\in R$. This leads to $q^{0}=qq'$, that is, $q=q^{0}$, which contradicts the condition $l_{pq}>1$. Therefore $\alpha q^{0}\notin R$. Now, if there exists $\gamma\in t(\alpha)^-$, with $\gamma\neq\alpha$, then $\gamma q^{0}\in R$, since $A$ is a special multiserial algebra. This implies $\gamma q=\delta\gamma'$, where $\delta=\gamma q^{0}\in R$. Therefore, $\gamma=p'p$, i.e., $\gamma=p$, which is impossible. Consequently, $|t(\alpha)^-|=1$.

Now, suppose that $t(\alpha)\neq t(p)$. Then, $p=p''\alpha p'$, for some paths $p''$ and $p'$, where $\alpha(p')^{0}\notin R$, since $p$ is nonzero. If there exists $\gamma\in t(\alpha)^-$, with $\gamma\neq \alpha$, then $\gamma(p')^{0}\in R$, since $A$ is a special multiserial algebra. However, since $(p,q)$ is a perfect pair, we have $\gamma p'=p'''p$, which leads to a contradiction. Therefore, $|t(\alpha)^-|=1$.
    
The proof for the case $|s(\beta)^+|=1$ follows by an analogous argument.
\end{proof}

\begin{corollary}\label{cor_2perfectpairs}
Let $A=\Bbbk Q/I$ be a monomial special multiserial algebra, with $R$ a minimal set of relations. Consider paths $p_1$, $p_2$ and $p_3$ such that $l_{p_2}\geq1$ and both $(p_1,\,p_2)$ and $(p_2,\,p_3)$ are perfect pairs. Then, for each vertex $i$ within the path $p_2$, there exists a unique incoming arrow and a unique outgoing arrow.
\end{corollary}
    \begin{proof}
It immediately follows from Lemma \ref{lemaSM_perfectpair}.
    \end{proof}
    
\begin{remark}
For a monomial special multiserial algebra $A$, if $(\alpha,\,\beta)$ is a perfect pair, where $\alpha$ and $\beta$ are arrows in $A$, then $|t(\alpha)^-|\leq 2$ and $|s(\beta)^+|\leq 2$. 

In fact,  for any $\gamma\in t(\alpha)^-$, we have two cases: if $\gamma\beta\in R$, it follows $\gamma=\alpha'\alpha$, impliying $\gamma=\alpha$, since $(\alpha,\,\beta)$ is a perfect pair. On contrary, if $\gamma\beta\notin R$, this implies $\gamma\neq\alpha$, since $A$ is special multiserial. The proof for the case $|s(\beta)^+|\leq 2$ is analogous.
\end{remark}

By \cite[Theorem 4.1]{chen2018gorenstein}, \textit{perfect paths} in a monomial algebra $A$ are uniquely associated with Gorenstein projective non-projective modules of $A$. This motivates the classification of perfect paths for specific classes of monomial algebras. For completeness, we recall the definition of a perfect path (cf. \cite[Definition 3.7]{chen2018gorenstein}). A nonzero path $p$ in a monomial algebra $A$ is called a \textit{perfect path} if there exists a sequence $p=p_1,\ldots,p_n,p_{n+1}=p$ of nonzero paths such that $(p_i,p_{i+1})$ is a perfect pair for all $1\leq i\leq n$. If the paths $p_i$ in the sequence are pairwise distinct, we call the sequence $p=p_1,\ldots,p_n,p_{n+1}=p$ a \textit{relation-cycle of} $p$.

The following result describes perfect paths when the monomial algebra $A$ is also special multiserial and UMP.

\begin{theorem}\label{perfectarrows}
Let $A=\Bbbk Q/I$ be a monomial special multiserial UMP algebra, with $R$ a minimal set of relations. Then any perfect path in $A$, if it exists, must take one of the following forms:
\begin{enumerate}
\item Arrows, or
\item Power of a cycle $u=u^0u^1\cdots u^{l_u}$ where
\begin{itemize}
\item There exists an integer $l\geq2$ such that $u^{(l)}\in R$.
\item $\beta u^0\notin R$ and $u^{l_u}\alpha\notin R$, for all arrows $\alpha,\,\beta$ in the quiver $Q$.
\end{itemize}
\end{enumerate}
\end{theorem}
\begin{proof} Let $p$ be a perfect path. By definition there exists a sequence $p=p_1,\ldots,p_n,p_{n+1}=p$ of nonzero paths such that $(p_i,p_{i+1})$ is a perfect pair for all $1\leq i\leq n$. Suppose there exists $i\in\{1,\cdots,n\}$ such that $p_i$ has length greater than or equal to two, i.e., $p_i$ is not an arrow. Then $p_{i-1}p_{i}$ and $p_{i}p_{i+1}$ are relations of length greater than two. By Remark \ref{remark_descripcion_relaciones}, it follows that $p_{i-1}p_{i}=p_{i}p_{i+1}$. This implies that $t(p_{i+1})=t(p_i)=s(p_{i+1})$ and $s(p_{i-1})=s(p_i)=t(p_{i-1})$. Consequently, $p_{i+1}$ and $p_{i-1}$ are powers of some cycles. 
Since $(p_i,\,p_{i+1})$ is a perfect pair, $p_{i+1}^{(2)}\notin R$. However, $p_i^{l_{p_i}}p_{i+1}\notin R$ as $p_i$ has length greater than or equal to two. This contradict the fact that $A$ is a special multiserial algebra. Therefore $p_{i+1}$ must have length greater than $1$. By similar arguments, we can show that each $p_j$, for $j\in\{1,\cdots,\,n\}$, is a power of some cycle and $l_{p_j}\geq1$. Since each path in the sequence starts and ends at the same vertex, there exists a cycle $u$ such that $p_ip_{i+1}=u^{(l)}$, for some integer $l\geq2$. Finally, observe that for any $i\in\{1,\cdots,n\}$, $\mathcal{R}(p_i)=\{p_{i+1}\}$ and $\mathcal{L}(p_{i+1})=\{p_i\}$, since $(p_i,\,p_{i+1})$ is a perfect pair. This implies that $\beta u^0\notin R$ and $u^{l_u}\alpha\notin R$, for all arrows $\alpha$ and $\beta$ in the quiver $Q$, completing the proof.
\end{proof}

\subsection{Main findings}\label{subsec:mainfind}
In \cite[Section 5]{chen2018gorenstein}, the authors provide a complete classification of the category of Gorenstein projective non-projective modules over quadratic monomial algebras, along with significant consequences. Their primary tool is the concept of \textit{quiver of relations}. Motivated by our characterization result, Theorem \ref{thm:umpquadratic}, this subsection aims to extend and generalize the main results and consequences of \cite{chen2018gorenstein} to the setting of monomial special multiserial UMP algebras. As a first step towards this goal, we introduce a generalized definition of the quiver of relations.

\begin{definition}\label{def:qrel} Let $A=\Bbbk Q/I$ be a monomial algebra, with $R$ is a minimal set of relations. The {\em left minimal paths relations quiver} $\mathcal{Q}^{\mathcal{L}}_A$ is defined as follows: 
\begin{itemize}
\item The vertices of $\mathcal{Q}^{\mathcal{L}}_A$ are the nontrivial paths $p\in Q$, with $p\notin I$, that satisfy one of the following conditions:
\begin{enumerate}
    \item There exists $q$ a nonzero path such that $t(q)=s(p)$ and $qp\in R$.
    \item There exists $q$ a nonzero path such that $s(q)=t(p)$ and $pq\in R$.
    \item $p$ is an arrow that satisfies neither of the two preceding conditions.
\end{enumerate}

\item An arrow $\mathcal{Q}^{\mathcal{L}}_A$ is of the form $[pq]:q\rightarrow p$, where $s(q)=t(p)$ and $p\in \mathcal{L}(q)$.
\end{itemize}
\end{definition}

Example \ref{ex3.4} in Subsection  \ref{subsec:exs} illustrates the construction of  $\mathcal{Q}^\mathcal{L}_A$.

\begin{remark}
 \begin{enumerate}
 \item If $p$ is an arrow satisfying condition $(3)$ in Definition \ref{def:qrel}, then $p$ is an isolated vertex in $\mathcal{Q}^{\mathcal{L}}_A$. To see this, suppose $p$ were not isolated. Then there would exist a vertex $q\in \mathcal{Q}^{\mathcal{L}}_A$ such that $p\in \mathcal{L}(q)$ or $q\in \mathcal{L}(p)$. By Lemma \ref{lemma_LR}, this implies that $pq'\in R$ for some $q'$ subpath of $q$ or $qp\in R$, respectively, which contradicts condition $(3)$.

\item For a quadratic monomial algebra $A$, the quiver $\mathcal{Q}^{\mathcal{L}}_A$ coincides with the relation quiver defined in \cite[Definition 5.2]{chen2018gorenstein}. Therefore, $\mathcal{Q}^{\mathcal{L}}_A$ generalizes the relation quiver to monomial algebras.
\end{enumerate}   
\end{remark}

Let $\mathcal{C}$ be a connected component of $\mathcal{Q}^\mathcal{L}_A$. We call $\mathcal{C}$ \emph{perfect} if it is a basic cycle and \emph{acyclic} if it contains no oriented cycles or loops.

By \cite[Lemma 5.3(2)]{chen2018gorenstein}, under the notation of Definition \ref{def:qrel}, for a quadratic monomial algebra $A$, $(p,\,q)$ is a perfect pair if and only if $[pq]$ is the unique arrow in the relation quiver with source $q$ and the unique arrow in the relation quiver with target $p$. However, as demonstrated by Example \ref{ex:qrdifference}, this characterization does not hold for monomial bound quiver algebras in general. 

Theorem \ref{lema5.3} below provides an extended characterization of $\mathcal{Q}^{\mathcal{L}}_A$ for monomial special multiserial UMP algebras $A$.

\begin{theorem}\label{lema5.3}
Let $A=\Bbbk Q/I$ be a monomial special multiserial UMP algebra and let $p$ be a nonzero path. Then:
    \begin{enumerate}
\item For any nonzero path $q$, with $s(q)=t(p)$, $(p,\,q)$ is a perfect pair if and only if there exists an arrow $[pq]: q\rightarrow p$ that is both the unique arrow in $\mathcal{Q}^{\mathcal{L}}_A$ starting at $q$ and the unique arrow ending at $p$, and only one of the following conditions hold:
\begin{enumerate}
\item There is no other arrow ending in a path $p'$, where $p=rp'$ and $r$ a nontrivial path;
\item There exists a cycle $u$, such that $pq=u^{(l)}$ for some $l\geq2$.
\end{enumerate}
\item The path $p$ is perfect if and only if its corresponding vertex in $\mathcal{Q}^{\mathcal{L}}_A$ belongs to a perfect component.
    \end{enumerate}
\end{theorem}
\begin{proof}
\begin{enumerate}
\item Suppose $(p,\,q)$ is a perfect pair. Then $pq\in R$, so there exists an arrow $[pq]: q\rightarrow p$ in $\mathcal{Q}^{\mathcal{L}}_A$. By Definition \ref{def:perpair}, $\mathcal{L}(q)=\{p\}$ obtaining that $[pq]$ is the unique arrow with source $q$. Now, suppose there exist nontrivial paths $p'$, $q'$ such that $p=rp'$ and $p'\in \mathcal{L}(q')$, i.e., there exists an arrow $[p'q']: q'\rightarrow p'$ in $\mathcal{Q}^{\mathcal{L}}_A$. By Lemma \ref{lemma_LR}, $p'q''\in R$, for some subpath $q''$ of $q'$. Consequently, $pq''=0$. Since $\mathcal{R}(p)=\{q\}$ by Definition \ref{def:perpair}, it follows that $q''=q\gamma$ for some path $\gamma$. We consider two cases. First, suppose that $\gamma$ is trivial. Then $q''=q$ and thus $p'q\in R$, i.e., $p'\in \mathcal{L}(q)=\{p\}$. Thus, $[pq]$ in $\mathcal{Q}^{\mathcal{L}}_A$ is the unique arrow ending in $p$ and there is no other arrow whose target is the path $p'$, where $p=rp'$ and $r$ is a nontrivial path. Now, suppose that $\gamma$ is nontrivial. Then, $p'q''$ generates a relation of length greater than two. Since $p'q\neq 0$ (otherwise there would be an arrow from $q$ to $p'$), $p'$ is a proper subpath of $p$ and hence $pq$ is a relation of length greater than two. Thus, $pq=\alpha p'q$ and $p'q''=p'q\gamma$, for some nonzero paths $\alpha$ and $\gamma$. 
By Remark \ref{remark_descripcion_relaciones}, it follows that $pq=p'q''$. Consequently, $t(\gamma)=t(q)=s(\gamma)$. Since $\gamma$ is non-trivial it is a cycle. Therefore, by Theorem \ref{thm:UMPmain}, $pq=u^{(l)}$, for some cycle $u$. Conversely, suppose there exists a unique arrow $[pq]: q\rightarrow p$ in $\mathcal{Q}^{\mathcal{L}}_A$ and there is no other arrow ending in a path $p'$, where $p=rp'$ and $r$ a nontrivial path. Then, $p\in\mathcal{L}(q)$, so $pq=0$. Now, consider a nonzero path $p'$ with $p'\in \mathcal{R}(q)$. By Lemma \ref{lemma_LR}, $p'q'\in R$, for some subpath $q'$ of $q$, with $q=q'r'$ for some path $r'$. By Definition \ref{def:qrel}, there exists an arrow from $q$ to $p'$. Our hypothesis implies that $p'=p$, i.e., $\mathcal{L}(q)=\{p\}$. Moreover, for any $q'\in\mathcal{R}(p)$ we have $p'q'\in R$ with $p=rp'$, by Lemma \ref{lemma_LR}. In particular, $p'\in\mathcal{L}(q')$. Thus, our hypothesis again implies that $p'=p$. Consequently, $pq'\in R$, so $[pq']$ is an arrow in $\mathcal{Q}^{\mathcal{L}}_A$. However, the unique arrow ending in $p$ is $[pq]$, which implies that $\mathcal{R}(p)=\{q\}$, i.e., $(p,\,q)$ is a perfect pair. The case where $pq=u^{(l)}$ for some cycle $u$ and $l\geq2$ can be handled similarly, also showing that $(p,\,q)$ is a perfect pair.

\item Suppose $p$ is a perfect path. By Theorem \ref{perfectarrows}, $p$ is either an arrow or a cycle. From part $(1)$ above, $p$ is a path within a perfect component of $\mathcal{Q}^{\mathcal{L}}_A$. Conversely, let $p$ be a path belonging to a perfect component of $\mathcal{Q}^{\mathcal{L}}_A$. Then, there exist paths $p_0,p_1,\ldots,p_n$ such that $[p_i p_{i+1}]$ is the unique arrow with source $p_{i+1}$ and target $p_i$, for $0\leq i \leq n-1$. In particular, $p_ip_{i+1}\in R$. Suppose there exists a nontrivial path $r_j$ such that $p_j=r_jp_j'$ and there is another arrow ending at $p_j'$, for some $0\leq j\leq n$. Then, $p_jp_{j+1}$ and $p_{j-1}p_j$ are relations of length greater than two. Thus, $p_{j+1}$ and $p_{j-1}$ are powers of cycles with length greater than one. Consequently, $p_ip_{i+1}$ coincides with a power of some cycle $u$, for all $0\leq i \leq n-1$. From part $(1)$ above, we conclude that $(p_i,\,p_{i+1})$ are perfect pairs for each $0\leq i \leq n-1$, that is, $p$ is a perfect path.
\end{enumerate}  
\end{proof}

Recall that a vertex $j$ in a quiver is \emph{bounded} if any path starting at $j$ has an upper bound (in terms of the its length).

\begin{lemma}\label{lema5.4}
Let $A=\Bbbk Q/I$ be a monomial special multiserial UMP algebra and let $p$ be a nonzero path. Suppose that $p$ corresponds to a vertex in $\mathcal{Q}^{\mathcal{L}}_A$. Then:
    \begin{enumerate}
\item The $A$-module $Ap$ is a Gorenstein-projective nonprojective module if and only if the vertex of $p$ in $\mathcal{Q}^{\mathcal{L}}_A$ belongs to a perfect component.
\item The $A$-module $Ap$ has finite projective dimension if and only if the vertex of $p$ in $\mathcal{Q}^{\mathcal{L}}_A$ is bounded.
\item If the vertex of $p$ in $\mathcal{Q}^{\mathcal{L}}_A$ is not bounded and is not within a perfect component, then for each $d\geq0$ the syzygy module $\Omega^d(Ap)$ is not Gorenstein-projective module.
    \end{enumerate}
\end{lemma}

\begin{proof}
\begin{enumerate}
\item If the vertex corresponding to $p$ is in a perfect component of $\mathcal{Q}^{\mathcal{L}}_A$, then $p$ is perfect by Theorem \ref{lema5.3}. By \cite[Proposition 4.4(4)]{chen2018gorenstein}, the $A$-module $Ap$ is Gorenstein-projective nonprojective. Conversely, if $Ap$ is a Gorenstein-projective nonprojective $A$-module, then by \cite[Proposition 4.4(2)]{chen2018gorenstein} there exists a unique perfect path $q$ such that $\mathcal{L}(p)=\{q\}$. Thus, there exists an arrow $[qp]$ in $\mathcal{Q}^{\mathcal{L}}_A$. By Theorem \ref{lema5.3}, $q$ belongs to a perfect component of $\mathcal{Q}^{\mathcal{L}}_A$, and therefore so does $p$.

\item It suffices to note that $\Omega(Ap)\cong\bigoplus\limits_{q\in \mathcal{L}(p)}Aq$ (see Lemma \ref{lemma3.1}) and recall the definition of the quiver $\mathcal{Q}^{\mathcal{L}}_A$.

\item Suppose the vertex corresponding to $p$ in $\mathcal{Q}^{\mathcal{L}}_A$ is unbounded and not in any perfect component. Assume, by contradiction, there is a $d\geq1$ such that its syzygy $A$-module $\Omega^d(Ap)$ is Gorenstein-projective. By part $(2)$ above, $\Omega^d(Ap)$ has infinite projective dimension and thus, $\Omega^d(Ap)$ is not projective. \cite[Proposition 4.4(4)]{chen2018gorenstein} then implies of the existence of a path $q$ such that $Aq\simeq \Omega^d(Ap)$ is Gorenstein-projective nonprojective. Consequently, there is a path of length $d$ from $p$ to $q$ in $\mathcal{Q}^{\mathcal{L}}_A$. By part $(1)$ above, $q$ is in some perfect component, which implies that $p$ belongs to the same perfect component, contradicting our assumption.
\end{enumerate} 
\end{proof}

The following lemma provides a well-known characterization of Gorenstein algebras (see, e.g., \cite[Theorem 3.2]{avramov2002absolute}) and is included here for completeness. Before stating the lemma, we recall some definitions from \cite[p. 1117-1118]{chen2018gorenstein}. An algebra $A$ is \emph{CM-free} if every Gorenstein projective $A$-module is projective. We call $A$ \emph{$d$-Gorenstein} (for $d\geq0$) if the injective dimension of $A$ as a (bimodule) $A$-module is at most $d$. The 0-Gorenstein algebras are precisely the self-injective algebras and 1-Gorenstein algebras are simply called Gorenstein algebras. The following result is well known (cf. \cite[Theorem 3.2]{avramov2002absolute}).

\begin{lemma}\label{lemma2.2}
 Let $A$ be a finite-dimensional algebra and let $d\geq0$. Then the algebra $A$ is $d$-Gorenstein if and only if, for each $A$-module $M$, the module $\Omega^d(M)$ is Gorenstein-projective.
\end{lemma}

\begin{proposition}\label{prop:homoUMP1} Let $A=\Bbbk Q/I$ be a monomial special multiserial UMP algebra. Denote by $d$ the length of the longest path in all the acyclic components of $\mathcal{Q}^{\mathcal{L}}_A$. Then,
\begin{enumerate}
    \item The algebra $A$ is CM-free if and only if the left minimal paths relations quiver $\mathcal{Q}^{\mathcal{L}}_A$ has no perfect components.
    \item The algebra $A$ is Gorenstein if and only if each connected component of the quiver $\mathcal{Q}^{\mathcal{L}}_A$ is either perfect or acyclic. In this case, the algebra $A$ is $(d + 2)$-Gorenstein.
    \item The algebra $A$ has finite global dimension if and only if each component of the relation quiver $\mathcal{Q}^{\mathcal{L}}_A$ is acyclic.
\end{enumerate} 
\end{proposition}
\begin{proof}
\begin{enumerate}
\item This is an immediately consequence of Theorem \ref{lema5.3} and correspondence Theorem in \cite[Theorem 4.1]{chen2018gorenstein}.

\item Suppose that $A$ is a Gorenstein algebra and there exists a path $p$ whose corresponding vertex in $\mathcal{Q}^{\mathcal{L}}_A$ is unbounded and not in any perfect component. By Lemma \ref{lema5.4}, each syzygy $\Omega^d(Ap)$ is not Gorenstein-projective. Then, by Lemma \ref{lemma2.2}, $A$ is not a Gorenstein algebra, which contradicting our hypothesis. Conversely, consider a path $p$ whose corresponding vertex in $\mathcal{Q}^{\mathcal{L}}_A$, belongs to a perfect or acyclic component. By Lemma \ref{lema5.4}, $Ap$ is either Gorenstein-projective or has projective dimension at most $d$. Consequently, $\Omega^d(Ap)$ is Gorenstein-projective. Now, if $q$ is a path such that $q=pq'$, for some nontrivial path $q'$, there exists an isomorphism $Ap\cong Aq$ mapping $x$ to $xq'$. Thus, $\Omega^d(Aq)$ is Gorenstein-projective for any nonzero path $q$. In general, by Lemma \ref{lemma3.1}, for any $A$-module $M$, $\Omega^2(M)$ is isomorphic to a direct summand of $Ap$, for some nonzero path $p$. As consequence of latter reasoning, $\Omega^{d+2}(M)$ is Gorenstein-projective. That is, the algebra $A$ is $(d+2)$-Gorenstein. 

\item By Lemma \ref{lemma2.2}, the algebra $A$ has finite global dimension if and only if is Gorenstein and CM-free. Thus, the conclusion follows directly from $(1)$ and $(2)$ above.
\end{enumerate}    
\end{proof}

Following \cite[Proposition 5.9]{chen2018gorenstein}, we aim to characterize the (stable) category of Gorenstein projective modules over a special multiserial, monomial, and UMP algebra $A$.  We first recall some notation from \cite[p. 1130]{chen2018gorenstein}.  Two perfect paths $p$ and $q$ in a monomial algebra $A$ are said to \emph{overlap} if they satisfy one of the following conditions:
\begin{enumerate}
\item[(O1)] $p=q$, and $p=p'x$ and $q=xq'$ for some nontrivial paths $x$, $p'$ and $q'$ such that the path $p'xq'$ is nonzero.
\item[(O2)] $p\neq q$, and $p = p'x$ and $q = xq'$ for some nontrivial path $x$ such that the path $p'xq'$ is nonzero.
\end{enumerate}

\begin{proposition}\label{prop:homoUMP2}
Let $A=\Bbbk Q/I$ be a monomial special multiserial UMP algebra. Let $d_1$, $d_2$, $\cdots,d_m$ be the lengths of the relation cycles in $A$, where relation cycles are identified up to cyclic permutation.  Then there is a triangulated equivalence $A\mbox{-}\underline{\mathrm{Gproj}}\simeq \mathcal{T}_{d_1}\times\cdots\times\mathcal{T}_{d_m}$ if and only if the only perfect paths in $A$ are arrows or at most quadratic powers of loops.
\end{proposition}
\begin{proof}By \cite[Proposition 5.9]{chen2018gorenstein}, such a triangulated equivalence exists if and only if A admits no overlaps. Theorem \ref{perfectarrows} states that perfect paths are either arrows or powers of cycles.  Thus, we only need to characterize the perfect paths that are cycles and admit overlaps, since arrows clearly do not overlap. By Theorem \ref{perfectarrows}, suppose there is a cycle $u$ such that $u^{(l)}\in R$, for some $l\geq 2$, and  $\beta u^{l_u}\notin R$, for all $\beta\in Q_1$. If $l_u\geq 1$, $u^{(1)}$ is a perfect path. Thus, defining $p=u^{(1)}$ and $q=u^1\cdots u^{l_u}u^0$, $p'=u^0=q'$ and $x=u^1\cdots u^{l_u}$, we have an overlap, since $p'xq'=u^{(1)}u^0\neq0$ is a proper subpath of $u^{(2)}$. Now, suppose that $u$ is a loop. Overlaps occur only when $l\geq4$. In this case, $u^{(2)}$ is a perfect path. Defining $p=u^{(2)}=q$ and $p'=x=q'=u^{(1)}$, we have an overlap, since $p'xq'=u^{(3)}\neq0$. If $l=2$, then the only perfect path defined by $u$ is $u^{(1)}$, which does not produce an overlap. If $l=3$, then the only perfect paths are $u^{(1)}$ and $u^{(2)}$, and neither produces an overlap, which completes the proof.
\end{proof}

\subsection{Examples and discussion}\label{subsec:exs}
Examples \ref{ex3.2} and \ref{ex3.3} show that the CM-free property is not ``local'', that is, if $A$ is an algebra with this property, for a weakly connected component $N$, the algebra $A_N$ does not inherit this property and viceversa. In particular, Theorem \ref{lema5.3} and Lemma \ref{lema5.4} may not extend to the more general context of locally monomial special multiserial UMP algebras (as stated in the classification theorem of \cite[Theorem 4.4]{CFR}).
\begin{example}\label{ex3.2}
Consider $A=\Bbbk Q/I$, where $Q$ is defined as in Figure \ref{fig:quiver1} and its ideal is $I=\langle\alpha_2\alpha_1, \alpha_1\alpha_2\rangle$.
\begin{figure}[h!]
    \centering
    $Q:\ \xymatrix{1 && 2 \ar[ll]^{\beta_2}\ar@/_1pc/[rr]_{\alpha_1} && 3 \ar@/_1pc/[ll]_{\alpha_2} \ar[rr]^{\beta_1} && 4}$
    \caption{Quiver example \ref{ex3.2}.}
    \label{fig:quiver1}
\end{figure}
 The algebra $A$ is not CM-free, since $\alpha_1$ and $\alpha_2$ are perfect paths, and CM-free algebras have no perfect paths (by \cite[Corollary 4.2]{chen2018gorenstein}).

Now, $\omega_{\alpha_i}=\alpha_i$ y $\omega_{\beta_i}=\beta_i$, with $i=1, 2$. Consequently, its weakly connected components are $N_i=\omega_{\alpha_i}\longrightarrow\omega_{\beta_i}$ and corresponding quiver is given by $2\xrightarrow{\A_1}3\xrightarrow{\beta_1}4$ and $3\xrightarrow{\A_2}2\xrightarrow{\beta_2}1$ whose ideal is $I_{N_i}=\langle0\rangle$. Thus, $A_{N_i}$ has no perfect paths and hence $A_{N_i}$ are CM-Free, for $i=1, 2$ (by \cite[Corollary 4.2]{chen2018gorenstein}).


\end{example}

\begin{example}\label{ex3.3}
Consider $A=\Bbbk Q/I$ where $Q$ is the quiver in the Figure \ref{fig:quiver2}. The ideal is $I=\langle\alpha_4\alpha_1\alpha_2\alpha_3,\,\alpha_2\alpha_3\alpha_4\alpha_1,\,\beta\alpha_4\rangle$. The algebra $A$ is CM-Free, as it has no perfect paths. Since $\omega_{\alpha_1}=\alpha_2\alpha_3\alpha_4\alpha_1$ y $\omega_{\beta}=\beta$ the quiver of the component is Figure \ref{fig:component}, whose ideal is $I_{N_{\alpha_1}}=\langle\alpha_4\alpha_1\alpha_2\alpha_3,\,\alpha_2\alpha_3\alpha_4\alpha_1\rangle$. Thus, the algebra $A_{N_{\alpha_1}}$  has perfect paths $\alpha_4\alpha_1$ and $\alpha_2\alpha_3$, and hence is not CM-free.
\begin{figure}[h!]
    \centering
    $Q:\ \ \xymatrix{ & & 3 \ar[dr]^{\alpha_1} &\\
 1\ar[r]^{\beta} & 2 \ar[ur]^{\alpha_4} & & 4 \ar[dl]^{\alpha_2}\\
 & & 5 \ar[ul]^{\alpha_3} &}$
    \caption{Quiver of Example \ref{ex3.3}.}
    \label{fig:quiver2}
\end{figure}

\begin{figure}[h!]
\centering
$Q_{N_{\alpha_1}}:\ \ \xymatrix{ & 3 \ar[dr]^{\alpha_1} & \\
 2 \ar[ur]^{\alpha_4} & & 4 \ar[dl]^{\alpha_2}\\
 & 5 \ar[ul]^{\alpha_3} & }$
\caption{Component $Q_{N_{\A_1}}$.}
\label{fig:component}
\end{figure}

\end{example}
Example \ref{ex3.4} illustrates the construction of our quiver $\mathcal{Q}^\mathcal{L}_A$ introduced in Definition \ref{def:qrel}.

\begin{example}\label{ex3.4}
Consider the algebra $A=\Bbbk Q/I$ whose quiver is defined in Figure \ref{fig:quiver3} and its ideal is $I=\langle (\alpha_2\alpha_3\alpha_1)^{(2)},\,\beta\alpha_3,\,\gamma_1\gamma_2\gamma_3\gamma_1\gamma_2,\,\gamma_2\beta\rangle$.
\begin{figure}[h!]
\centering
$Q:\ \ \xymatrix{2 \ar[dr]^{\gamma_2} & & & & 5 \ar[dd]^{\alpha_1} \\
 & 3 \ar[rr]^{\beta} \ar[dl]^{\gamma_3} & & 4 \ar[ur]^{\alpha_3} & \\
1 \ar[uu]^{\gamma_1} & & & & 6 \ar[ul]^{\alpha_2} }$
\caption{Quiver of Example \ref{ex3.4}.}
\label{fig:quiver3}
\end{figure}

The quiver $\mathcal{Q}^\mathcal{L}_A$ is given by the components in Figure \ref{fig:quiver5}.
\vspace{5mm}

\begin{figure}[h!]
\centering
$\xymatrix@=5mm{ \alpha_1 \ar[r] & \alpha_2\alpha_3\alpha_1\alpha_2\alpha_3 &
\alpha_3\alpha_1 \ar[r] & \alpha_2\alpha_3\alpha_1\alpha_2 &
\alpha_2\alpha_3\alpha_1 \ar@(l,u)[] &
\alpha_1\alpha_2\alpha_3\alpha_1 \ar[r] & \alpha_2\alpha_3\\ \\
\alpha_3 \ar[r] & \beta \ar[r] & \gamma_2 \ar[r] & \gamma_1\gamma_2\gamma_3\gamma_1 \ar[r] & \gamma_1\gamma_2\gamma_3 \ar@(l,u)[] & \gamma_1\gamma_2 \ar[l] & \gamma_3\gamma_1\gamma_2 \ar[l]\\ \\
& \alpha_3\alpha_1\alpha_2\alpha_3\alpha_1 \ar[r] & \alpha_2 & \gamma_2\gamma_3\gamma_1\gamma_2 \ar[r] & \gamma_1 & \gamma_3. &
}$
\caption{Quiver $\mathcal{Q}^\mathcal{L}_A$ of Example \ref{ex3.4}.}
\label{fig:quiver5}
\end{figure}
\end{example}

Example \ref{ex:qrdifference} illustrates the construction of $\mathcal{Q}^\mathcal{L}_A$ and highlights a key difference between this quiver and the relation quiver defined in \cite{chen2018gorenstein}.

 \begin{example} \label{ex:qrdifference} Consider the quiver in Figure \ref{fig:quiver3}
 with relations $R=\{\gamma_1\gamma_2\gamma_3,\,\gamma_2\beta,\,\A_2\A_3\}$. The quiver $\mathcal{Q}^\mathcal{L}_A$ looks like:
 
$$\xymatrix{\beta \ar[r] & \gamma_2\hspace{0.5cm}\A_3 \ar[r] & \A_2\hspace{0.5cm}\gamma_3 \ar[r] & \gamma_1\gamma_2\hspace{0.5cm}\gamma_2\gamma_3 \ar[r] & \gamma_1\hspace{0.5cm}\A_1.}$$

Clearly, $(\gamma_2,\,\beta),\,(\A_2,\,\A_3),\,(\gamma_1,\,\gamma_2\gamma_3)$ are perfect pairs. However, since $\mathcal{R}(\gamma_1\gamma_2)=\{\gamma_3,\,\beta\}$ we have that the pair $(\gamma_3,\gamma_1\gamma_2)$ is not perfect.
\end{example}

\subsection*{Funding}
The first and fourth authors were partially supported by CODI (Universidad de Antioquia, UdeA) by project numbers 2020-33305 and 2023-62291, respectively. The third author gratefully acknowledges the funding provided by CODI (Universidad de Antioquia, UdeA) through its postdoctoral program, project number 2022-52654.


\end{document}